\documentclass[11pt]{amsart}

\usepackage{amsmath, amssymb, amsthm, mathtools}
\usepackage{geometry}
\usepackage{enumitem}
\usepackage{hyperref}
\usepackage{tikz}
\usepackage{float}
\usepackage{placeins}
\usepackage{graphicx}
\usepackage{subcaption}
\usepackage[utf8]{inputenc}
\usepackage{stmaryrd}

\usetikzlibrary{trees, decorations.markings, calc}

\newtheorem{theorem}{Theorem}[section]
\newtheorem{lemma}[theorem]{Lemma}

\theoremstyle{definition}

\theoremstyle{remark}
\newtheorem{remark}[theorem]{Remark}

\newcommand{\Z}{\mathbb{Z}}
\newcommand{\Q}{\mathbb{Q}}

\tikzset{
  vtx/.style={circle, fill=black, inner sep=2pt},
  arr/.style={->, thick}
}

\title[Plumbed $3$--Manifolds and Neumann Moves]{Plumbed $3$--Manifolds and Neumann Moves: \\
From Weakly Negative Definite to Negative Definite}

\begin{document}
\author{Noah Pope}
\maketitle
\begingroup
\renewcommand\thefootnote{}
\endgroup
\begin{abstract}
We give a constructive proof that every weakly negative definite plumbing tree can be transformed into a negative definite one by a finite sequence of Neumann moves. The argument combines Neumann's plumbing calculus with the diagonalization algorithm of Duchon, Eisenbud, and Neumann, which extracts the eigenvalues of the framing matrix directly from the combinatorics of the tree. We show that any positive eigenvalues are supported on linear branches and can be eliminated systematically via controlled applications of Neumann moves. This provides an explicit algorithm reducing weakly negative definite plumbing trees to negative definite ones.
\end{abstract}
\bigskip

\section{Introduction}

The motivation for this work originates in a remark of Gukov, Katzarkov, and
Svoboda~\cite{GukovKatzarkovSvoboda}, where it is stated in footnote 6 that any
weakly negative definite plumbing tree can be transformed into a negative definite one
by a sequence of Neumann moves, although no proof is provided. A proof of
this statement appears later in the appendix of Harichurn, N{\'e}methi,
and Svoboda~\cite{HarichurnNemethiSvoboda}. The goal of this paper is to present a
more detailed and fully constructive version of this proof. In particular, we
give an explicit algorithm showing how any weakly negative definite plumbing
tree can be converted into a negative definite one via a finite sequence of
Neumann moves.

\begin{theorem}[Main Theorem]
Every weakly negative definite plumbed $3$--manifold is negative definite.
Equivalently, every weakly negative definite plumbing tree can be transformed
into a negative definite one by a finite sequence of Neumann moves.
\end{theorem}

To prove the Main Theorem, we begin by reviewing the necessary background on
plumbing trees and their associated framing matrices, as well as the Neumann
moves and the diagonalization algorithm, which extracts the eigenvalues of the
framing matrix directly from the graph. The proof then proceeds by analyzing the
presence of positive eigenvalues and showing how they can be eliminated through
a finite sequence of Neumann moves. This reduction is carried out case by case,
ultimately yielding a negative definite plumbing tree.

Negative definite plumbing trees have been studied extensively. Classically, by
a theorem of Grauert~\cite{Grauert62}, a plumbed $3$--manifold arises as the
link of a normal complex surface singularity if and only if its intersection
matrix is negative definite, providing a direct bridge between low-dimensional
topology and singularity theory. More recently, negative definite plumbings
have played a central role in the study of quantum invariants: the
$\widehat{Z}$--invariant of Gukov, Pei, Putrov, and Vafa~\cite{GPPV} was
originally defined on the broader class of weakly negative definite plumbings,
but most subsequent work from Gukov, Katzarkov, and
Svoboda~\cite{GukovKatzarkovSvoboda} and Harichurn, N{\'e}methi, and
Svoboda~\cite{HarichurnNemethiSvoboda}---restricts to negative definite graphs
both for analytic reasons (convergence of the defining $q$-series) and for
access to tools such as lattice cohomology. The reduction
proved here justifies that restriction: the broader class of weakly negative
definite plumbing trees represents the same family of $3$--manifolds as the
negative definite ones.
\section{Preliminaries}

\subsection{Plumbing Trees and Framing Matrices}

Plumbing graphs and their associated intersection matrices are classical objects
in the study of plumbed $3$--manifolds; see, for example, Neumann~\cite{Neumann79}.
More generally, plumbing constructions and their applications to quantum invariants
and $3$--manifold topology are discussed in, for instance, Gukov and
Manolescu~\cite{GM21}. In this paper, we restrict attention to the case
where the underlying graph is a finite tree, since this is the setting relevant
for our main result.

A \emph{plumbing tree} $\Gamma$ is a finite tree whose vertices are equipped with
integer weights. We choose and fix an ordering of the vertex set and write
$$
V(\Gamma)=\{v_1,\dots,v_s\},
$$
where $s = |V(\Gamma)|$. For each $i$, let $m_i\in\Z$ denote the weight of the
vertex $v_i$, called its \emph{Euler number}. We write $E(\Gamma)$ for the edge
set of $\Gamma$.

The \emph{framing matrix} associated to $\Gamma$ is the symmetric $s\times s$
matrix
$$
B := (B_{ij})_{i,j=1}^s,
\qquad
B_{ij} =
\begin{cases}
m_i, & i=j, \\
1, & i\neq j \text{ and } \{v_i,v_j\}\in E(\Gamma), \\
0, & \text{otherwise}.
\end{cases}
$$

The plumbing tree $\Gamma$ (or equivalently, the matrix $B$) is called
\emph{negative definite} if all eigenvalues of $B$ are strictly negative.
The tree $\Gamma$ is called \emph{weakly negative definite} if the framing
matrix $B$ is invertible over $\Q$ and the restriction of $B^{-1}$ to the
subspace of $\Q^s$ spanned by the vertices of degree at least three is
negative definite; that is, all eigenvalues of this restriction are strictly
negative.

\subsection{Neumann Moves}

The local transformations used in plumbing calculus were introduced by
Neumann~\cite{Neumann79}. Each move replaces a local
configuration in a plumbing tree with another one and induces a
homeomorphism of the associated plumbed $3$--manifolds.

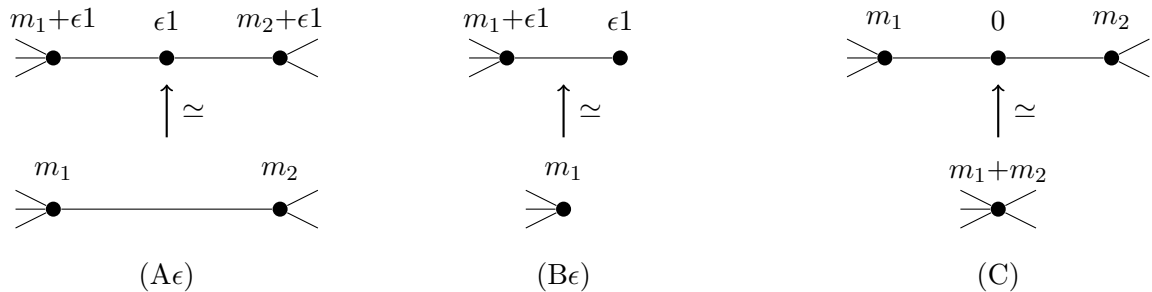
\begin{figure}[htbp]
\centering
\tikzset{
  vtx/.style={circle, fill=black, inner sep=2pt},
  leg/.style={thin},
  edge/.style={thin}
}
\begin{tikzpicture}[scale=1.0]

\begin{scope}[xshift=-5.5cm]

\node at (0,   0.5) {$m_1{+}\epsilon1$};
\node at (1.5, 0.5) {$\epsilon1$};
\node at (3,   0.5) {$m_2{+}\epsilon1$};

\node[vtx] (a1) at (0,0)   {};
\node[vtx] (a2) at (1.5,0) {};
\node[vtx] (a3) at (3,0)   {};
\draw[edge] (a1)--(a2)--(a3);

\draw[leg] (a1)--++(-0.5, 0.25);
\draw[leg] (a1)--++(-0.5, 0);
\draw[leg] (a1)--++(-0.5,-0.25);

\draw[leg] (a3)--++(0.5, 0.25);
\draw[leg] (a3)--++(0.5,-0.25);

\draw[->, thick] (1.5,-1.05) -- (1.5,-0.35);
\node[right] at (1.57,-0.7) {$\simeq$};

\node at (0,  -1.5) {$m_1$};
\node at (3,  -1.5) {$m_2$};

\node[vtx] (b1) at (0,-2.0)  {};
\node[vtx] (b2) at (3,-2.0)  {};
\draw[edge] (b1)--(b2);

\draw[leg] (b1)--++(-0.5, 0.25);
\draw[leg] (b1)--++(-0.5, 0);
\draw[leg] (b1)--++(-0.5,-0.25);

\draw[leg] (b2)--++(0.5, 0.25);
\draw[leg] (b2)--++(0.5,-0.25);

\node at (1.5,-2.9) {(A$\epsilon$)};
\end{scope}

\begin{scope}[xshift=0.5cm]

\node at (0,   0.5) {$m_1{+}\epsilon1$};
\node at (1.5, 0.5) {$\epsilon1$};

\node[vtx] (c1) at (0,0)   {};
\node[vtx] (c2) at (1.5,0) {};
\draw[edge] (c1)--(c2);

\draw[leg] (c1)--++(-0.5, 0.25);
\draw[leg] (c1)--++(-0.5, 0);
\draw[leg] (c1)--++(-0.5,-0.25);

\draw[->, thick] (0.75,-1.05) -- (0.75,-0.35);
\node[right] at (0.82,-0.7) {$\simeq$};

\node at (0.75,-1.5) {$m_1$};
\node[vtx] (d1) at (0.75,-2.0) {};

\draw[leg] (d1)--++(-0.5, 0.25);
\draw[leg] (d1)--++(-0.5, 0);
\draw[leg] (d1)--++(-0.5,-0.25);

\node at (0.75,-2.9) {(B$\epsilon$)};
\end{scope}

\begin{scope}[xshift=5.5cm]

\node at (0,   0.5) {$m_1$};
\node at (1.5, 0.5) {$0$};
\node at (3,   0.5) {$m_2$};

\node[vtx] (e1) at (0,0)   {};
\node[vtx] (e2) at (1.5,0) {};
\node[vtx] (e3) at (3,0)   {};
\draw[edge] (e1)--(e2)--(e3);

\draw[leg] (e1)--++(-0.5, 0.25);
\draw[leg] (e1)--++(-0.5, 0);
\draw[leg] (e1)--++(-0.5,-0.25);

\draw[leg] (e3)--++(0.5, 0.25);
\draw[leg] (e3)--++(0.5,-0.25);

\draw[->, thick] (1.5,-1.05) -- (1.5,-0.35);
\node[right] at (1.57,-0.7) {$\simeq$};

\node at (1.5,-1.5) {$m_1{+}m_2$};
\node[vtx] (f1) at (1.5,-2.0) {};

\draw[leg] (f1)--++(-0.5, 0.25);
\draw[leg] (f1)--++(-0.5, 0);
\draw[leg] (f1)--++(-0.5,-0.25);

\draw[leg] (f1)--++(0.5, 0.25);
\draw[leg] (f1)--++(0.5,-0.25);

\node at (1.5,-2.9) {(C)};
\end{scope}

\end{tikzpicture}
\caption{The Neumann moves on plumbing trees. Here $\epsilon \in \{+,-\}$.}
\label{fig:Neumann}
\end{figure}

The move $(A\epsilon)$ collapses a vertex of weight $\epsilon1$ and degree $2$
lying between two adjacent vertices. The move $(B\epsilon)$ removes a
terminal vertex of weight $\epsilon1$ and degree $1$. The move $(C)$
collapses a vertex of weight $0$ and degree $2$, merging the two
adjacent vertices into a single vertex whose weight is the sum of
their original weights.

The Neumann moves have the following effect on the eigenvalues of the framing
matrix $B$. The \emph{signature} of an invertible symmetric matrix $B$ is the
pair
$$
(n_+(B),\, n_-(B)),
$$
where $n_+(B)$ and $n_-(B)$ denote the numbers of positive and negative
eigenvalues of $B$, counted with multiplicity.

\begin{lemma}
\label{lem:neumann-signature}
For each Neumann move, let $B$ denote the framing matrix of the bottom graph
and $B_\circ$ the framing matrix of the top graph in Figure~\ref{fig:Neumann}.
The signatures are related as follows.
\begin{enumerate}
\item Move $(A-)$:
\[
n_+(B_\circ) = n_+(B)
\qquad\mbox{and}\qquad
n_-(B_\circ) = n_-(B) + 1.
\]
\item Move $(A+)$:
\[
n_+(B_\circ) = n_+(B) + 1
\qquad\mbox{and}\qquad
n_-(B_\circ) = n_-(B).
\]
\item Move $(B-)$:
\[
n_+(B_\circ) = n_+(B)
\qquad\mbox{and}\qquad
n_-(B_\circ) = n_-(B) + 1.
\]
\item Move $(B+)$:
\[
n_+(B_\circ) = n_+(B) + 1
\qquad\mbox{and}\qquad
n_-(B_\circ) = n_-(B).
\]
\item Move $(C)$:
\[
n_+(B_\circ) = n_+(B) + 1
\qquad\mbox{and}\qquad
n_-(B_\circ) = n_-(B) + 1.
\]
\end{enumerate}
\end{lemma}

\begin{proof}
See Neumann~\cite{Neumann79}
and Gompf and Stipsicz~\cite[Chapter~5]{GompfStipsicz}.
\end{proof}

\section{The Diagonalization Algorithm}
\label{sec:diag}

We summarize the diagonalization algorithm due to Duchon, originally appearing in
his doctoral dissertation, and presented in detail by Eisenbud and
Neumann~\cite{EisenbudNeumann,DuchonThesis}. Given a plumbing tree $\Gamma$ with
framing matrix $B$, the algorithm produces a diagonal matrix that is similar to
$B$. In particular, the signs of the diagonal entries record the signature of~$B$.

More precisely, the algorithm proceeds by recursively removing edges of $\Gamma$
through local simplification moves, modifying the vertex weights in the process
while allowing, in general, rational weights. When all edges have been removed, the
resulting graph consists of isolated vertices with weights
$$
d_1,\dots,d_s,
$$
yielding a diagonal matrix
$$
D = \operatorname{diag}(d_1,\dots,d_s).
$$
The values $d_1,\dots,d_s$ are exactly the eigenvalues of the framing matrix $B$,
repeated with multiplicity.

\subsection{Local Simplification Moves}

Choose a vertex of the plumbing tree and orient all edges toward it.
The orientation is used only to organize the application of the algorithm and does
not affect the resulting eigenvalues.

The local simplification moves act exactly like elementary row--column operations
on the framing matrix and preserve its eigenvalues. We describe the two cases
that occur in the diagonalization algorithm.

\subsubsection*{Case 1: A vertex whose incoming edges all come from leaves with nonzero weights}

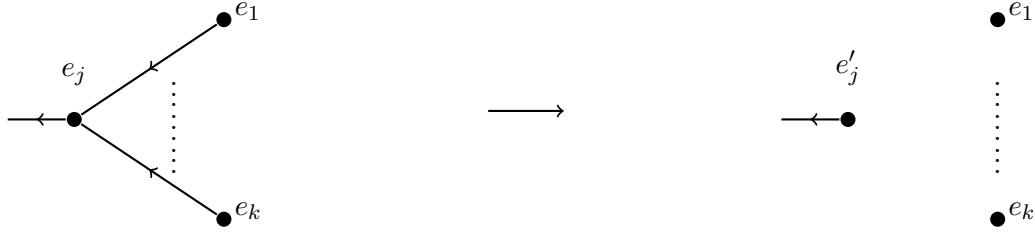
\FloatBarrier
\begin{center}
\begin{minipage}{0.38\textwidth}
\centering
\begin{tikzpicture}[scale=1.1, thick,
  vtx/.style={circle, fill=black, inner sep=2pt},
  lbl/.style={above, yshift=2mm},
  midarrow/.style={
    postaction={decorate},
    decoration={markings, mark=at position 0.5 with {\arrow{>}}}
  }
]

\node[vtx,label={[lbl]{$e_j$}}] (ej) at (0,0) {};

\draw[midarrow] (ej) -- (-0.8,0);

\node[vtx,label={[right]{$e_1$}}] (e1) at (1.8,1.2) {};
\draw[midarrow] (e1) -- (ej);

\node[vtx,label={[right]{$e_k$}}] (ek) at (1.8,-1.2) {};
\draw[midarrow] (ek) -- (ej);

\node at (1.2,0.4) {$\vdots$};
\node at (1.2,0.0) {$\vdots$};
\node at (1.2,-0.4) {$\vdots$};

\end{tikzpicture}
\end{minipage}
\hfill
\begin{minipage}{0.12\textwidth}
\centering
\begin{tikzpicture}
\draw[line width=0.8pt,->] (0,0) -- (1.0,0);
\end{tikzpicture}
\end{minipage}
\hfill
\begin{minipage}{0.38\textwidth}
\centering
\begin{tikzpicture}[scale=1.1, thick,
  vtx/.style={circle, fill=black, inner sep=2pt},
  lbl/.style={above, yshift=2mm},
  midarrow/.style={
    postaction={decorate},
    decoration={markings, mark=at position 0.5 with {\arrow{>}}}
  }
]

\node[vtx,label={[lbl]{$e_j'$}}] (ejp) at (0,0) {};

\draw[midarrow] (ejp) -- (-0.8,0);

\node[vtx,label={[right]{$e_1$}}] (e1p) at (1.8,1.2) {};
\node[vtx,label={[right]{$e_k$}}] (ekp) at (1.8,-1.2) {};

\node at (1.8,0.4) {$\vdots$};
\node at (1.8,0.0) {$\vdots$};
\node at (1.8,-0.4) {$\vdots$};

\end{tikzpicture}
\end{minipage}
\end{center}

The weight of the adjacent vertex is modified according to
$$
e_j'
=
e_j
-
\frac{1}{e_1}
-
\cdots
-
\frac{1}{e_k},
$$
where $e_1,\dots,e_k$ are the weights of the leaves attached to the
same parent. This move reduces the number of edges.

\begin{remark}
Case~1 can only be applied when all descendant leaves have
nonzero weights. If a $0$--weighted leaf is adjacent, the formula
$$
e_j' = e_j - \sum_i \frac{1}{e_i}
$$
is not defined. In this situation, Case~2 must be applied instead to eliminate the
$0$--weighted leaf.
\end{remark}

\paragraph{Example (Case 1).}

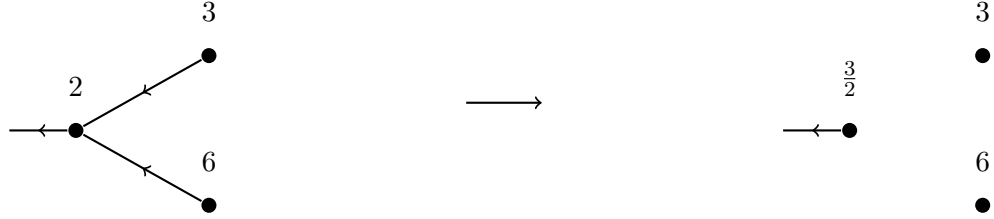
\FloatBarrier
\begin{center}
\begin{minipage}{0.38\textwidth}
\centering
\begin{tikzpicture}[scale=1.1, thick,
  vtx/.style={circle, fill=black, inner sep=2pt},
  lbl/.style={above, yshift=2mm},
  midarrow/.style={
    postaction={decorate},
    decoration={markings, mark=at position 0.5 with {\arrow{>}}}
  }
]

\node[vtx,label={[lbl]{$2$}}] (ej) at (0,0) {};

\draw[midarrow] (ej) -- (-0.8,0);

\node[vtx,label={[lbl]{$3$}}] (e1) at (1.6,0.9) {};
\node[vtx,label={[lbl]{$6$}}] (e2) at (1.6,-0.9) {};

\draw[midarrow] (e1) -- (ej);
\draw[midarrow] (e2) -- (ej);

\end{tikzpicture}
\end{minipage}
\hfill
\begin{minipage}{0.12\textwidth}
\centering
\begin{tikzpicture}
\draw[line width=0.8pt,->] (0,0) -- (1.0,0);
\end{tikzpicture}
\end{minipage}
\hfill
\begin{minipage}{0.38\textwidth}
\centering
\begin{tikzpicture}[scale=1.1, thick,
  vtx/.style={circle, fill=black, inner sep=2pt},
  lbl/.style={above, yshift=2mm},
  midarrow/.style={
    postaction={decorate},
    decoration={markings, mark=at position 0.5 with {\arrow{>}}}
  }
]

\node[vtx,label={[lbl]{$\tfrac{3}{2}$}}] (ejp) at (0,0) {};

\draw[midarrow] (ejp) -- (-0.8,0);

\node[vtx,label={[lbl]{$3$}}] (f1) at (1.6,0.9) {};
\node[vtx,label={[lbl]{$6$}}] (f2) at (1.6,-0.9) {};

\end{tikzpicture}
\end{minipage}
\end{center}

\subsubsection*{Case 2: A leaf of weight $0$}

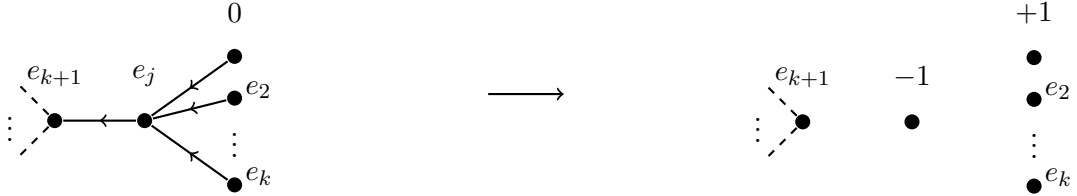
\FloatBarrier
\begin{center}
\begin{minipage}{0.38\textwidth}
\centering
\begin{tikzpicture}[scale=0.85, thick,
  vtx/.style={circle, fill=black, inner sep=2pt},
  lbl/.style={above, yshift=2mm},
  midarrow/.style={
    postaction={decorate},
    decoration={markings, mark=at position 0.5 with {\arrow{>}}}
  }
]


\node[vtx,label={[lbl]{$e_{k+1}$}}] (ekp) at (-1.4,0) {};
\draw[dashed] (ekp) -- ++(-0.6,0.6);
\draw[dashed] (ekp) -- ++(-0.6,-0.6);
\node at (-2.1,0) {$\vdots$};

\node[vtx,label={[lbl]{$e_j$}}] (ej) at (0,0) {};

\draw[midarrow] (ej) -- (ekp);

\node[vtx,label={[lbl]{$0$}}] (z) at (1.4,1.0) {};
\draw[midarrow] (z) -- (ej);

\node[vtx,label={[right]{$e_2$}}] (e2) at (1.4,0.35) {};
\node[vtx,label={[right]{$e_k$}}] (ek) at (1.4,-1.0) {};
\draw[midarrow] (e2) -- (ej);
\draw[midarrow] (ek) -- (ej);

\node at (1.4,-0.25) {$\vdots$};

\end{tikzpicture}
\end{minipage}
\hfill
\begin{minipage}{0.12\textwidth}
\centering
\begin{tikzpicture}
\draw[line width=0.8pt,->] (0,0) -- (1.0,0);
\end{tikzpicture}
\end{minipage}
\hfill
\begin{minipage}{0.38\textwidth}
\centering
\begin{tikzpicture}[scale=0.85, thick,
  vtx/.style={circle, fill=black, inner sep=2pt},
  lbl/.style={above, yshift=2mm}
]


\node[vtx,label={[lbl]{$e_{k+1}$}}] (ekp2) at (-1.4,0) {};
\draw[dashed] (ekp2) -- ++(-0.6,0.6);
\draw[dashed] (ekp2) -- ++(-0.6,-0.6);
\node at (-2.1,0) {$\vdots$};

\node[vtx,label={[lbl]{$-1$}}] at (0.3,0) {};

\node[vtx,label={[lbl]{$+1$}}] at (2.2,1.0) {};

\node[vtx,label={[right]{$e_2$}}] at (2.2,0.35) {};
\node[vtx,label={[right]{$e_k$}}] at (2.2,-1.0) {};

\node at (2.2,-0.25) {$\vdots$};

\end{tikzpicture}
\end{minipage}
\end{center}

If a leaf has weight $0$, the parent vertex $e_j$ is
replaced by a vertex of weight $-1$, and the $0$--leaf becomes a $+1$--leaf. The weights of all
other vertices remain unchanged.

\paragraph{Example (Case 2).}
We illustrate Case~2 of the diagonalization algorithm on a minimal plumbing tree
containing a single $0$--weighted leaf.

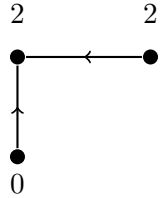
\begin{center}
\begin{tikzpicture}[scale=1.1, thick,
  vtx/.style={circle, fill=black, inner sep=2pt},
  lbl/.style={above, yshift=2mm},
  blbl/.style={below, yshift=-2mm},
  midarrow/.style={
    postaction={decorate},
    decoration={markings, mark=at position 0.5 with {\arrow{>}}}
  }
]

\node[vtx,label={[lbl]{$2$}}] (ej) at (0,0) {};

\node[vtx,label={[blbl]{$0$}}] (z) at (0,-1.2) {};
\draw[midarrow] (z) -- (ej);

\node[vtx,label={[lbl]{$2$}}] (e2) at (1.6,0) {};
\draw[midarrow] (e2) -- (ej);

\end{tikzpicture}
\end{center}

Since one of the leaves adjacent to the central vertex has weight $0$, Case~1
cannot be applied. We therefore apply Case~2.

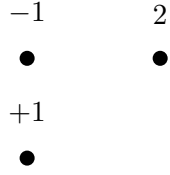
\begin{center}
\begin{tikzpicture}[scale=1.1, thick,
  vtx/.style={circle, fill=black, inner sep=2pt},
  lbl/.style={above, yshift=2mm}
]

\node[vtx,label={[lbl]{$-1$}}] at (0,0) {};

\node[vtx,label={[lbl]{$+1$}}] at (0,-1.2) {};

\node[vtx,label={[lbl]{$2$}}] at (1.6,0) {};

\end{tikzpicture}
\end{center}

In this example, the parent vertex of weight $2$ is replaced by a vertex of
weight $-1$, and the $0$--leaf becomes a $+1$--leaf, while the weight of the other adjacent
vertex remains unchanged. The resulting isolated vertices have weights
$$
-1,\quad +1,\quad 2,
$$
which are the eigenvalues of the framing matrix~$B$ for this plumbing tree.

\subsection{Continued Fractions and Contractible Branches}

Following Moore--Tarasca~\cite{MooreTarasca}, we recall the relevant
definitions and results on contractible branches. A \emph{branch} of a
plumbing tree $\Gamma$ is a maximal linear subpath connecting a vertex of
degree at least three to a leaf through a sequence of degree-two vertices.
A branch is called \emph{contractible} if it can be collapsed to a single
vertex by a finite sequence of Neumann moves.

Contractible branches can be characterized using continued fractions. For
integers $a_0, \dots, a_n$ with $a_n \neq 0$, define the
\emph{negative continued fraction}
$$
\llbracket a_0, \dots, a_n \rrbracket
\;:=\;
a_0 - \cfrac{1}{a_1 - \cfrac{1}{\ddots - \cfrac{1}{a_n}}}
\;\in\mathbb{Q}.
$$
When $a_n = 0$, one sets
$$
\llbracket a_0, \dots, a_{n-2}, a_{n-1}, 0 \rrbracket
\;:=\;
\llbracket a_0, \dots, a_{n-2} \rrbracket.
$$
For a path $\Gamma_{wv}$ from a vertex $w$ of degree at least $3$ to a leaf $v$
through a sequence of degree-$2$ vertices, let $m_w, u_1, \dots, u_n, m_v$ be
the weights of the vertices from $w$ to $v$ taken in this order.
\begin{lemma}
\label{lem:moretarasca1}
Changing $\Gamma_{wv}$ through a sequence of Neumann moves preserves
$\llbracket m_w, u_1, \dots, u_n, m_v \rrbracket$.
\end{lemma}

\begin{lemma}
\label{lem:moretarasca2}
The path $\Gamma_{wv}$ is contractible if and only if
$\llbracket m_w, u_1, \dots, u_n, m_v \rrbracket$ is an integer. In
this case, $\llbracket m_w, u_1, \dots, u_n, m_v \rrbracket$ equals the
weight of the vertex resulting from the contraction.
\end{lemma}

By Lemma~\ref{lem:moretarasca1}, applying Neumann moves along a branch preserves
its continued fraction value. In particular, for a branch with weights
$m_w, u_1, u_2, \dots, u_n, m_v$ oriented toward a central vertex, the
diagonalization algorithm reduces the branch step by step. At each stage, the
outermost edge is eliminated. The successive reductions produce the values
$$
\llbracket u_1, \dots, u_n, m_v \rrbracket, \quad
\llbracket u_2, \dots, u_n, m_v \rrbracket, \quad \dots, \quad
\llbracket u_n, m_v \rrbracket, \quad m_v,
$$
as weights of the resulting isolated vertices. By Lemma~\ref{lem:moretarasca2},
the central vertex weight is then updated by subtracting
$1/\llbracket m_w, u_1,\dots,u_n, m_v\rrbracket$ from its original weight for
each adjacent incoming branch. Contracting a branch via Neumann moves does not affect the
eigenvalues supported on other branches or on the high-degree vertices, a property
that will be essential in the proof of the Main Theorem.

\subsection*{Termination}

Both Case~1 and Case~2 strictly reduce the number of edges in the plumbing tree.
Since the graph is finite, the algorithm must terminate after finitely many
steps. At termination, all edges have been removed and the plumbing tree has
become a disconnected graph consisting entirely of isolated vertices with
weights
$$
d_1,\dots,d_s.
$$

The framing matrix of a disconnected plumbing graph is diagonal, with diagonal
entries given by these vertex weights. Consequently, the final weights
$d_1,\dots,d_s$ are exactly the eigenvalues of the original framing matrix~$B$.
Thus, the diagonalization algorithm provides a direct combinatorial procedure
for reading off the signature of~$B$ from the weights of the resulting
isolated vertices.
\FloatBarrier

\section{Proof of the Main Theorem}

Throughout the proof, $B$ denotes the framing matrix associated to~$\Gamma$.
We begin with the key lemma, which shows that weak negative definiteness
forces every high-degree vertex to contribute a strictly negative eigenvalue
after the diagonalization algorithm is applied.

\begin{lemma}
\label{lem:highdeg-negative}
Let $\Gamma$ be a weakly negative definite plumbing tree with framing matrix $B$.
After applying the diagonalization algorithm, any positive eigenvalue of $B$
arises from a vertex of degree at most $2$ in $\Gamma$.
\end{lemma}

\begin{proof}
This directly follows from the definition of weakly negative definite. 
\end{proof}

We now carry out the elimination of positive eigenvalues case by case according
to the structure of $\Gamma$. In each case we orient all edges toward a chosen
high-degree vertex (or an arbitrary vertex if none exists) before applying the
diagonalization algorithm; the resulting eigenvalues do not depend on this choice.

\subsection{Path Case: Linear tree with a positive leaf weight}
\label{sec:branch-case}

Suppose that $\Gamma$ is a plumbing tree all of whose vertices have
degree at most $2$, so that $\Gamma$ is a path. Label the vertices
$v_1, \dots, v_r$ and their integer weights $m_1, \dots, m_r$.
After applying the diagonalization algorithm, the graph becomes disconnected and
each vertex contributes a weight, which we denote
$$
d_1,\dots,d_r,
$$
and which are precisely the eigenvalues of the framing matrix~$B$. If all $d_i < 0$, then the plumbing tree is already negative definite
and no further modification is required. Assume instead that there exists at
least one positive eigenvalue $d_i$ corresponding to a leaf. In this case
$d_i = m_i$ from the diagonalization algorithm. Apply the sequence of Neumann
moves
$$
(B+)^{-1}(A-)^{\,n},
$$
where $n := m_i - 1$. Each iteration of the move $(A-)$ introduces a new vertex adjacent to the leaf. By Lemma~\ref{lem:neumann-signature}, each application of
$(A-)$ preserves the number of positive eigenvalues and introduces one
additional negative eigenvalue, while the final application of $(B+)^{-1}$
removes a positive eigenvalue. By Lemma~\ref{lem:moretarasca1}, these moves preserve the eigenvalue $d_j$ with $j \neq i$. Consequently, this sequence
eliminates the positive eigenvalue $d_i$ without introducing any new positive
eigenvalues. Repeating this procedure for each positive leaf weight
yields a plumbing tree whose framing matrix has no positive eigenvalues
corresponding to leaves. After diagonalization, however, positive continued
fraction values may still appear at interior vertices of the path; the next
section treats this case.

\medskip
\noindent\textbf{Example.}
We illustrate this procedure with an explicit example. At each stage, we record
the signature $(n_+, n_-)$ of the framing matrix, where $n_+$ and $n_-$ denote
the numbers of positive and negative eigenvalues respectively.

\FloatBarrier
\begin{figure}[htbp]
\centering
\resizebox{\textwidth}{!}{%
\begin{tikzpicture}[x=2.4cm, y=-0.8cm]
\tikzset{
  vtx/.style={circle, fill=black, inner sep=2pt},
  arr/.style={->, thick}
}


\node[vtx,label=right:{$5$}]  at (0,2)   {};
\node[vtx,label=right:{$0$}]  at (0,3)   {};
\draw (0,2)--(0,3);
\node[font=\small] at (0,4.2) {$(1,1)$};
\draw[arr] (0.4,2.5) -- node[above,font=\small] {$A-$} (0.6,2.5);

\node[vtx,label=right:{$4$}]  at (1,1.5) {};
\node[vtx,label=right:{$-1$}] at (1,2.5) {};
\node[vtx,label=right:{$-1$}] at (1,3.5) {};
\draw (1,1.5)--(1,3.5);
\node[font=\small] at (1,4.7) {$(1,2)$};
\draw[arr] (1.4,2.5) -- node[above,font=\small] {$A-$} (1.6,2.5);

\node[vtx,label=right:{$3$}]  at (2,1)   {};
\node[vtx,label=right:{$-1$}] at (2,2)   {};
\node[vtx,label=right:{$-2$}] at (2,3)   {};
\node[vtx,label=right:{$-1$}] at (2,4)   {};
\draw (2,1)--(2,4);
\node[font=\small] at (2,5.2) {$(1,3)$};
\draw[arr] (2.4,2.5) -- node[above,font=\small] {$A-$} (2.6,2.5);

\node[vtx,label=right:{$2$}]  at (3,0.5) {};
\node[vtx,label=right:{$-1$}] at (3,1.5) {};
\node[vtx,label=right:{$-2$}] at (3,2.5) {};
\node[vtx,label=right:{$-2$}] at (3,3.5) {};
\node[vtx,label=right:{$-1$}] at (3,4.5) {};
\draw (3,0.5)--(3,4.5);
\node[font=\small] at (3,5.7) {$(1,4)$};
\draw[arr] (3.4,2.5) -- node[above,font=\small] {$A-$} (3.6,2.5);

\node[vtx,label=right:{$1$}]  at (4,0)   {};
\node[vtx,label=right:{$-1$}] at (4,1)   {};
\node[vtx,label=right:{$-2$}] at (4,2)   {};
\node[vtx,label=right:{$-2$}] at (4,3)   {};
\node[vtx,label=right:{$-2$}] at (4,4)   {};
\node[vtx,label=right:{$-1$}] at (4,5)   {};
\draw (4,0)--(4,5);
\node[font=\small] at (4,6.2) {$(1,5)$};
\draw[arr] (4.4,2.5) -- node[above,font=\small] {$B+$} (4.6,2.5);

\node[vtx,label=right:{$-2$}] at (5,0.5) {};
\node[vtx,label=right:{$-2$}] at (5,1.5) {};
\node[vtx,label=right:{$-2$}] at (5,2.5) {};
\node[vtx,label=right:{$-2$}] at (5,3.5) {};
\node[vtx,label=right:{$-1$}] at (5,4.5) {};
\draw (5,0.5)--(5,4.5);
\node[font=\small] at (5,5.7) {$(0,5)$};

\end{tikzpicture}%
}
\caption{Path case example. Successive $(A-)$ moves decrease the leaf weight
by one while introducing negative eigenvalues; the final $(B+)$ eliminates the
positive eigenvalue. Signatures $(n_+, n_-)$ record the evolution of the
framing matrix.}
\label{fig:branch-case-example}
\end{figure}
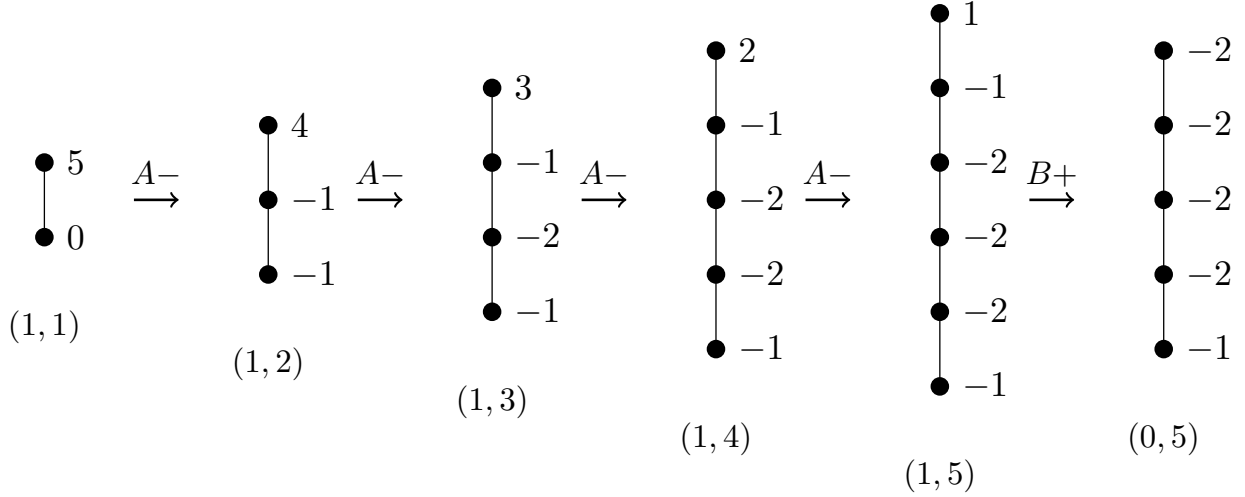
\FloatBarrier

\subsection{Interior Path Case: Positive weight at a non-leaf vertex}
\label{sec:interior-branch-case}

Suppose the plumbing tree $\Gamma$ is a path with weights $m_1, \dots, m_r$
and that after applying the diagonalization algorithm, the first $n$ vertices
contribute negative continued fraction values $d_1, \dots, d_n < 0$, but the
$(n+1)$-th vertex contributes a positive continued fraction value $d_{n+1} > 0$.
We proceed by induction on $n$ to show that this positive value can be eliminated
by a finite sequence of Neumann moves.

\medskip
\noindent\textbf{Base case.}
When $n = 0$, the positive continued fraction value $d_1$ is at a terminal
vertex, and the result follows from Section~\ref{sec:branch-case}.

\medskip
\noindent\textbf{Inductive step.}
Assume the result holds whenever the positive continued fraction value occurs at
position $\leq n$. Suppose $d_{n+1} > 0$. Apply the sequence of Neumann moves
$$
(A+)^{-1}(A-)^{k},
$$
where $k := m_{n+1} - 1$. Each iteration of the move $(A-)$ introduces a new vertex before the vertex $v_{n+1}$. Each application of $(A-)$ reduces $m_{n+1}$ by
exactly $1$, so after $k$ applications the updated value satisfies $m_{n+1} = 1$.
By Lemma~\ref{lem:neumann-signature}, each application of $(A-)$ increases $n_-$
by $1$ but leaves $n_+$ unchanged, so the number of positive eigenvalues remains
constant throughout.
Once $m_{n+1} = 1$, a single application of $(A+)^{-1}$
decreases $n_+$ by $1$, eliminating the positive eigenvalue entirely.
By Lemma~\ref{lem:moretarasca1}, these moves preserve the eigenvalues of the rest of the tree. The vertex
at position $n+1$ now contributes a negative eigenvalue. If the continued
fraction value at any subsequent vertex also requires correction, the same
procedure applies to it afterwards. The inductive hypothesis then guarantees
termination after finitely many steps.

\medskip
\noindent\textbf{Example.}
We illustrate the output of the diagonalization algorithm applied to the
path $(-2, 3, c)$, where $-2$ and $3$ are the weights of the first two vertices
and $c$ is left general. The value of $c$ does not affect the elimination of the
positive eigenvalue corresponding to the internal vertex weighted $3$; if the
weight at the third vertex also requires correction, the same procedure applies
to it afterwards. We apply $(A-)$ twice followed by a single $(A+)^{-1}$. The
sign pattern beneath each diagram uses $+$ and $-$ for eigenvalues that are
positive and negative respectively; the symbol $*$ denotes the eigenvalue
contributed by the vertex with weight $c$, whose sign is not determined by this
step and is handled afterwards.
\begin{figure}[H]
\centering
\tikzset{
  vtx/.style={circle, fill=black, inner sep=2pt},
  arr/.style={->, thick}
}
\begin{tikzpicture}[x=3.5cm, y=-0.8cm]

\node[vtx,label=right:{$-2$}] at (0,0) {};
\node[vtx,label=right:{$3$}]  at (0,1) {};
\node[vtx,label=right:{$c$}]  at (0,2) {};
\draw (0,0)--(0,2);
\node[font=\small] at (0,3.2) {$(-,+,*)$};
\draw[arr] (0.4,1) -- node[above,font=\small] {$A-$} (0.6,1);

\node[vtx,label=right:{$-3$}] at (1,0) {};
\node[vtx,label=right:{$-1$}]  at (1,1) {};
\node[vtx,label=right:{$2$}] at (1,2) {};
\node[vtx,label=right:{$c$}] at (1,3) {};
\draw (1,0)--(1,3);
\node[font=\small] at (1,4.2) {$(-,-,+,*)$};
\draw[arr] (1.4,1.5) -- node[above,font=\small] {$A-$} (1.6,1.5);

\node[vtx,label=right:{$-3$}] at (2,0) {};
\node[vtx,label=right:{$-2$}]  at (2,1) {};
\node[vtx,label=right:{$-1$}] at (2,2) {};
\node[vtx,label=right:{$1$}] at (2,3) {};
\node[vtx,label=right:{$c$}] at (2,4) {};
\draw (2,0)--(2,4);
\node[font=\small] at (2,5.2) {$(-,-,-,+,*)$};
\draw[arr] (2.4,2) -- node[above,font=\small] {$(A+)^{-1}$} (2.6,2);

\node[vtx,label=right:{$-3$}] at (3,0) {};
\node[vtx,label=right:{$-2$}] at (3,1) {};
\node[vtx,label=right:{$-2$}] at (3,2) {};
\node[vtx,label=right:{$c$}] at (3,3) {};
\draw (3,0)--(3,3);
\node[font=\small] at (3,4.2) {$(-,-,-,*)$};

\end{tikzpicture}
\caption{Interior path case example. The positive weight
at the interior vertex is reduced by successive $(A-)$ moves until it
reaches $1$, then the vertex is eliminated by $(A+)^{-1}$. The value of $c$ is left
general and handled afterwards if needed. Lemma~\ref{lem:moretarasca1} ensures
weak negative definiteness is preserved throughout.}
\label{fig:interior-branch-case}
\end{figure}
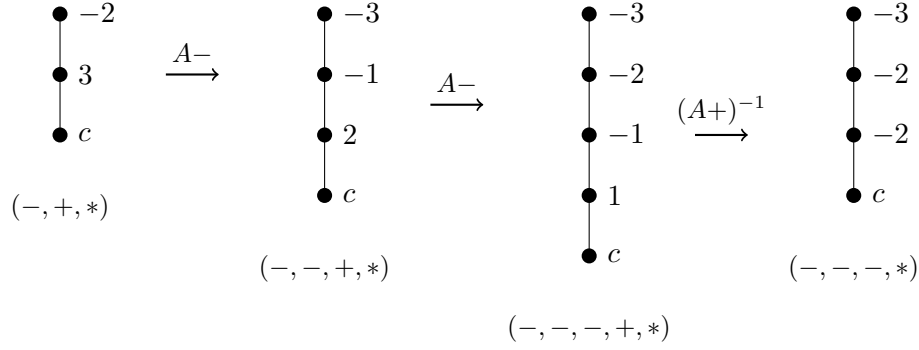
\FloatBarrier

\subsection{Star Case: Exactly one vertex of degree $\ge 3$}
\label{sec:star-case}

Suppose that the plumbing tree $\Gamma$ contains exactly one vertex $x$ of degree
at least $3$. We refer to $x$ as the \emph{central vertex} and write
$b_1,\dots,b_k$ for the branches attached to $x$, where $k\ge 3$. By
Lemma~\ref{lem:highdeg-negative}, $x$ contributes a strictly negative eigenvalue
after diagonalization. All positive eigenvalues therefore arise from the branches.

\medskip
\noindent\textbf{Elimination of positive branch eigenvalues.}
Any positive branch eigenvalue is supported on a vertex of degree $1$ or $2$ and
may be eliminated independently using the path case or interior path case
procedure as appropriate. The central vertex remains untouched throughout and
continues to contribute a negative eigenvalue. Repeating for each positive branch
eigenvalue eliminates all positivity in the spectrum of the framing matrix.

\medskip
\noindent\textbf{Example.}
We illustrate the star case with a simple three-armed example.
Figure~\ref{fig:star-case-diag} shows the orientation and application of the
diagonalization algorithm; Figure~\ref{fig:star-case-neumann} shows the
subsequent Neumann moves. The signature $(n_+, n_-)$ is recorded beneath each
diagram.

\begin{figure}[t]
\centering
\tikzset{
  vtx/.style={circle, fill=black, inner sep=2pt},
  midarrow/.style={
    postaction={decorate},
    decoration={markings, mark=at position 0.5 with {\arrow{>}}}
  }
}
\begin{tikzpicture}[scale=0.75]

\node[vtx,label=above:{$2$}]          (t)  at (0,3)  {};
\node[vtx,label=below:{$-2$}]         (c)  at (0,0)  {};
\node[vtx,label=left:{$-1$}]          (l)  at (-2,0) {};
\node[vtx,label=right:{$-1$}]         (r)  at (2,0)  {};
\draw[midarrow] (t)--(c);
\draw[midarrow] (l)--(c);
\draw[midarrow] (r)--(c);
\node[font=\small] at (0,-1.2) {(a)};

\draw[->,thick] (3,1.5)--node[above,font=\small]{diag.}(5,1.5);

\node[vtx,label=above:{$2$}]                    (t2) at (8,3)  {};
\node[vtx,label=above:{$-\tfrac{1}{2}$}]        (c2) at (8,0)  {};
\node[vtx,label=above:{$-1$}]                   (l2) at (6,0)  {};
\node[vtx,label=above:{$-1$}]                   (r2) at (10,0) {};
\node[font=\small] at (8,-1.2) {$(1,3)$};
\node[font=\small] at (8,-1.9) {(b)};

\end{tikzpicture}
\caption{Star case example, part 1. (a) Three-armed star with edges oriented
toward the central vertex. (b) After the diagonalization algorithm: all edges
removed, yielding isolated vertices with eigenvalues $2, -\tfrac{1}{2}, -1, -1$,
signature $(1,3)$.}
\label{fig:star-case-diag}
\end{figure}
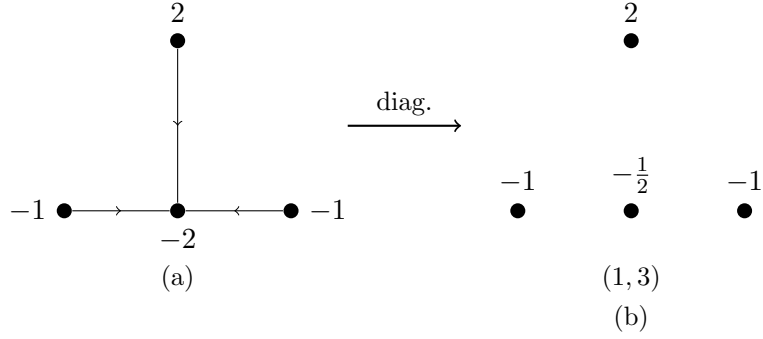

\begin{figure}[t]
\centering
\tikzset{
  vtx/.style={circle, fill=black, inner sep=2pt}
}
\begin{tikzpicture}[scale=0.75]

\node[vtx,label=above:{$2$}]                   (t)  at (0,3)   {};
\node[vtx,label=below:{$-2$}]                  (c)  at (0,0)   {};
\node[vtx,label=left:{$-1$}]                   (l)  at (-2,0)  {};
\node[vtx,label=right:{$-1$}]                  (r)  at (2,0)   {};
\draw (t)--(c); \draw (l)--(c); \draw (r)--(c);
\node[font=\small] at (0,-1.2) {$(1,3)$};
\node[font=\small] at (0,-1.9) {(c)};

\draw[->,thick] (3,1.5)--node[above,font=\small]{$A-$}(5,1.5);

\node[vtx,label=above:{$1$}]                   (t2) at (8,3)   {};
\node[vtx,label=right:{$-1$}]                  (m2) at (8,1.5) {};
\node[vtx,label=below:{$-3$}]                  (c2) at (8,0)   {};
\node[vtx,label=left:{$-1$}]                   (l2) at (6,0)   {};
\node[vtx,label=right:{$-1$}]                  (r2) at (10,0)  {};
\draw (t2)--(m2)--(c2); \draw (l2)--(c2); \draw (r2)--(c2);
\node[font=\small] at (8,-1.2) {$(1,4)$};
\node[font=\small] at (8,-1.9) {(d)};

\draw[->,thick] (11,1.5)--node[above,font=\small]{$(B+)^{-1}$}(13,1.5);

\node[vtx,label=above:{$-2$}]                  (t3) at (16,3)  {};
\node[vtx,label=below:{$-3$}]                  (c3) at (16,0)  {};
\node[vtx,label=left:{$-1$}]                   (l3) at (14,0)  {};
\node[vtx,label=right:{$-1$}]                  (r3) at (18,0)  {};
\draw (t3)--(c3); \draw (l3)--(c3); \draw (r3)--(c3);
\node[font=\small] at (16,-1.2) {$(0,4)$};
\node[font=\small] at (16,-1.9) {(e)};

\end{tikzpicture}
\caption{Star case example, part 2. (c) Original graph with signature $(1,3)$.
(d) After $(A-)$: signature $(1,4)$. (e) After $(B+)$: all eigenvalues negative,
signature $(0,4)$.}
\label{fig:star-case-neumann}
\end{figure}
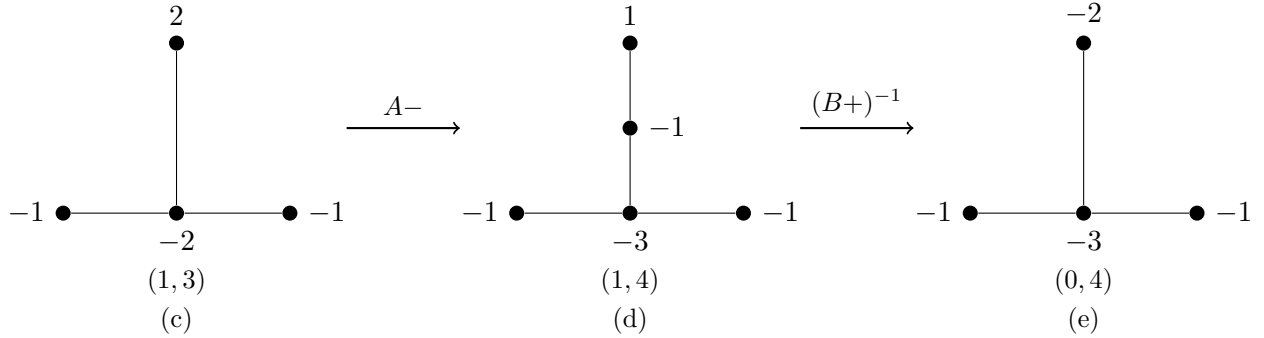

After eliminating the positive branch eigenvalue, all eigenvalues of the framing
matrix are strictly negative. Hence, the resulting plumbing tree is negative
definite, completing the proof in the star case.

\subsection{General Case: Multiple vertices of degree $\ge 3$}

We now consider the general case, in which the plumbing tree $\Gamma$ contains
more than one vertex of degree at least $3$. Let $\{x_1,\dots,x_r\}$ denote the
set of all such vertices and let $B$ be the associated framing matrix.

Orient each edge of $\Gamma$ toward the closest vertex among the $\{x_i\}$, with
ties broken arbitrarily. Vertices lying in linear chains between high-degree
vertices or at the ends of such chains are handled by the branch case. By
Lemma~\ref{lem:highdeg-negative}, every positive eigenvalue of $B$ arises from a
vertex of degree at most $2$. Each such eigenvalue lies either on a linear chain
or on a branch attached to some $x_i$. If it occurs at a terminal vertex, it is
eliminated by the path case sequence $(B+)^{-1}(A-)^{\,n}$ for an appropriate $n$. If
it occurs at an interior vertex, it is eliminated by the interior path case
procedure $(A+)^{-1}(A-)^{k}$ for an appropriate $k$. In either situation the
procedure introduces no new positive eigenvalues and does not affect the signature
at any $x_i$.

\medskip
\noindent\textbf{Conclusion.}
Repeating the elimination procedure for each positive eigenvalue yields a plumbing
tree in which all eigenvalues of the framing matrix are strictly negative. Hence
every weakly negative definite plumbing tree can be transformed into a negative
definite one by a finite sequence of Neumann moves. This completes the proof of
the Main Theorem.


\begin{thebibliography}{99}

\footnotetext{This paper has been submitted in partial fulfillment of the
requirements for the master's degree at Virginia Commonwealth University.}

\footnotetext{Supported in part by NSF Grant DMS-2404896 (PI: Nicola Tarasca).}

\bibitem{AJK}
R.~Akhmechet, P.~K. Johnson, and V.~Krushkal,
\emph{Lattice cohomology and $q$-series invariants of $3$--manifolds},
Journal f{\"u}r die reine und angewandte Mathematik (Crelle's Journal)
\textbf{798} (2023), 269--294.

\bibitem{DuchonThesis}
N.~Duchon,
\emph{Th\`ese de doctorat},
University of Maryland, 1982.

\bibitem{EisenbudNeumann}
D.~Eisenbud and W.~D. Neumann,
\emph{Three-Dimensional Link Theory and Invariants of Plane Curve Singularities},
Annals of Mathematics Studies, Vol.~110,
Princeton University Press, Princeton, NJ, 1985.

\bibitem{GompfStipsicz}
R.~E. Gompf and A.~I. Stipsicz,
\emph{4--Manifolds and Kirby Calculus},
Graduate Studies in Mathematics, Vol.~20,
American Mathematical Society, Providence, RI, 1999.

\bibitem{Grauert62}
H.~Grauert,
\emph{{\"U}ber Modifikationen und exzeptionelle analytische Mengen},
Mathematische Annalen \textbf{146} (1962), 331--368.

\bibitem{GukovKatzarkovSvoboda}
S.~Gukov, L.~Katzarkov, and J.~Svoboda,
\emph{$\widehat{Z}$ and splice diagrams},
Symmetry, Integrability and Geometry: Methods and Applications
\textbf{21} (2025), 073.

\bibitem{GM21}
S.~Gukov and C.~Manolescu,
\emph{A two-variable series for knot complements},
Quantum Topology \textbf{12} (2021), no.~1.

\bibitem{GPPV}
S.~Gukov, D.~Pei, P.~Putrov, and C.~Vafa,
\emph{$BPS$ spectra and $3$--manifold invariants},
Journal of Knot Theory and Its Ramifications \textbf{29} (2020), no.~2, 2040003.

\bibitem{HarichurnNemethiSvoboda}
S.~Harichurn, A.~N{\'e}methi, and J.~Svoboda,
\emph{$\Delta$ invariants of plumbed manifolds},
arXiv:2412.02042 [math.GT], 2024.

\bibitem{MooreTarasca}
A.~H. Moore and N.~Tarasca,
\emph{Root lattices and invariant series for plumbed $3$--manifolds},
arXiv:2405.14972 [math.GT], 2024.

\bibitem{Neumann79}
W.~D. Neumann,
\emph{An invariant of plumbed homology spheres},
in Topology Symposium, Siegen 1979 (U.~Koschorke and W.~D. Neumann, eds.),
Lecture Notes in Mathematics, vol.~788, Springer, Berlin, 1980, pp.~125--144.

\end{thebibliography}
\end{document}